\documentclass[12pt,titlepage]{amsart}
\usepackage{amsopn}
\usepackage{amssymb}
\usepackage{amsaddr}
\usepackage{graphicx}
\usepackage[all]{xy}
\usepackage{amscd}
\usepackage{color}
\usepackage[colorlinks]{hyperref}
\usepackage{amsmath,amsthm,amsfonts,amssymb}
\usepackage{verbatim}
\usepackage{latexsym}

\newtheorem{thm}{Theorem}[section]
 \newtheorem{cor}[thm]{Corollary}
 
 \newtheorem{prop}[thm]{Proposition}

 \newtheorem{exam}[thm]{Example}
 
 \theoremstyle{definition}
 \theoremstyle{remark}

\begin{document}
\title[Iteration of a Multiplication
operator] {Frame properties induced by Iteration of a
Multiplication operator on Hardy spaces}

\author{Jahangir Cheshmavar$^*$}

\address{Department of Mathematics,
Payame Noor University, P.O.Box: 19395-4697, Tehran, Iran \\
\emph{j$_{_-}$cheshmavar@pnu.ac.ir}\\
}


\keywords{frame, multiplication operator, Hardy space, dynamical
sampling, maximum modulus principle, open mapping
theorem.\\
\indent $^{*}$ Corresponding author}
 \dedicatory{} \commby{}
\begin{abstract}
Motivated by the study of frame properties arising from iterates
of linear operators, it was previously established that the
multiplication operator $T_{\phi}x(t)=\phi(t)x(t)$ cannot generate
a frame in $L^2(a,b)$ (Results Math, 2019). In this note, we
examine the behavior of such operators on the Hardy space
$H^2(\mathbb{D})$, the Hilbert space of holomorphic functions on
the open unit disk $\mathbb{D}$ with square-integrable boundary
values. We show that, in contrast to the $L^2$-setting, the
iterates of $T_\phi$, for $\phi \in H^\infty(\mathbb{D})$, exhibit
fundamentally different frame properties in $H^2(\mathbb{D})$,
leading to new structural insights and results.
\end{abstract}
\maketitle

\section{\textbf{Introduction and Preliminaries}}\label{Sec1}
Let $\mathbb{T}$ be the unit circle, defined by $\{z \in
\mathbb{C}: |z|=1\}$ and let $\mathbb{D}$ is the open unit disk,
defined by $\{z \in \mathbb{C}: |z|<1\}$. The set of all bounded
holomorphic functions of $\mathbb{D}$ will be denoted by
$H^{\infty}(\mathbb{D})$. For any $\phi \in
H^{\infty}(\mathbb{D})$, define
\begin{eqnarray}
\|\phi\|_{\infty}:=\sup_{z \in \mathbb{D}}|\phi(z)|,
\end{eqnarray}
then $H^{\infty}(\mathbb{D})$ is a Banach space. A function $\phi
\in H^{\infty}(\mathbb{D})$ is inner if its radial boundary values
satisfy $|\phi(e^{i\theta})|=1$ a.e. on $\mathbb{T}$. Now let
$H^2(\mathbb{D})$ is the Hardy space on $\mathbb{D}$. By
definition, $H^2(\mathbb{D})$ consists of holomorphic functions
$f$ on $\mathbb{D}$ with finite norm
\begin{eqnarray}
\| f \|_{2}^2 := \sup_{0 < r < 1} \int_0^{2\pi} |f(r
e^{i\theta})|^2 \frac{d\theta}{2\pi} < \infty.
\end{eqnarray}
Equivalently, for $f(z) = \sum_{n=0}^\infty a_n z^n$, the norm is
$\|f\|_{2}^2 = \sum_{n=0}^\infty |a_n|^2 < \infty$. Thus,
$H^2(\mathbb{D})$ is a Hilbert space isomorphic to
$\ell^2(\mathbb{N}_0)$, with the monomials $\{z^n\}_{n=0}^\infty$
forming an orthonormal basis. For any $f,g \in H^2(\mathbb{D})$,
the inner product is
\begin{eqnarray}
\langle f, g \rangle := \lim_{r \to 1^-} \int_0^{2\pi} f(r
e^{i\theta}) \overline{g(r e^{i\theta})} \frac{d\theta}{2\pi}.
\end{eqnarray}
If $f$ and $g$ have radial limits (which exists almost everywhere
on $\mathbb{T}$ for $H^2$-functions), we consider
\begin{eqnarray}
\langle f, g \rangle =\int_0^{2\pi} f( e^{i\theta}) \overline{g(
e^{i\theta})} \frac{d\theta}{2\pi}.
\end{eqnarray}
If $\phi \in H^{\infty}(\mathbb{D})$, the multiplication operator
$T_{\phi}$ associated with $\phi$ is defined on $H^2(\mathbb{D})$
by
\begin{eqnarray*}
T_\phi:H^2(\mathbb{D}) \to H^2(\mathbb{D}), \,\ T_\phi
f(z)=\phi(z) f(z),
\end{eqnarray*}
ensuring $T_\phi$ is bounded. The operator norm is $\| T_\phi \| =
\| \phi \|_{\infty} = \sup_{z \in \mathbb{D}} |\phi(z)|$. For
$T_\phi f = \phi f$ we have
\begin{eqnarray*}
\langle g, T_{\phi}^n f \rangle=\langle g, \phi^n f \rangle=
\int_0^{2\pi} g(r e^{i\theta}) \overline{\phi(r e^{i\theta})^n f(r
e^{i\theta})} \frac{d\theta}{2\pi}.
\end{eqnarray*}
Recall that, a sequence $\{T_\phi^n f\}_{n=0}^\infty = \{\phi^n
f\}_{n=0}^\infty \subset H^2(\mathbb{D})$ is a frame, if it
satisfies the frame condition, i.e., there exists constants $A,
B>0$, called the lower and upper frame bounds, so that
\begin{eqnarray*}\label{form1}
A \|g\|^2 \leq \sum_{n=0}^\infty |\langle g, \phi^n f \rangle|^2
\leq B \|g\|^2, \,\ \forall g \in H^2(\mathbb{D}),
\end{eqnarray*}
for some $A, B > 0$. Given a frame $\{\phi^n f\}_{n=0}^\infty$ in
$H^2(\mathbb{D})$, the \textit{frame operator} $S:H^2(\mathbb{D})
\rightarrow H^2(\mathbb{D})$ is defined by
\begin{eqnarray*}
Sg=\sum_{n=0}^{\infty}\langle g, \phi^n f\rangle \phi^n f, \,\
\forall g \in H^2(\mathbb{D}).
\end{eqnarray*}
The concept of representing frames through iteration methods is
closely connected to the field of dynamical sampling. This
important area focuses on understanding how well families of
elements generated by iterates of certain operators form frames
within a Hardy space. Dynamical sampling plays a significant role
in at least three key domains: signal processing and control, see
e.g., \cite{Aldroubi.2018} explores frame-based recovery of
evolving signals, emphasizing noise resilience and actuator
modelings; frames reconstruction via operators, as demonstrated in
\cite{Cabrelli.2020,Christensen.2017,Cheshmavar}, only specific
frames admit representations through operator iterates, providing
insight into the structural constraints of such frames in Hilbert
spaces; operator theory and functional analysis, studies like
\cite{Aldroubi.2016,Christensen.2017,Cheshmavar.2023} reveal deep
connections between spectral decomposition, operator dynamics, and
frame representation via operators.

In the study of operator theory on Hardy spaces, the relationship
between multiplication operators and the frame properties of
associated function systems offers profond structural insights,
for example, \cite{Seip}, examines frames in Hardy and Bergman
spaces, providing a rigorous framework for sampling and connecting
frame conditions to the geometry of zero sets and the behavior of
inner functions. Similarly, the authors in \cite{Aguillera}
investigate frames of iterations, linking the existence of frame
to the spectral properties of the operator $T_{\phi}$ and its
realization within vector-valued Hardy spaces.

A natural question arises concerning the conditions under which
the iterated system $\{ T_\phi^n f \}_{n=0}^\infty = \{ \phi^n f
\}_{n=0}^\infty$ forms a frame for $H^2(\mathbb{D})$. Recently,
authors in \cite{Kasumov.2019} investigated the frame properties
of sequences generated by iterates of a multiplication operator
acting on $L^2(a,b)$, We refer the reader to \cite{Christensen,
Aldroubi.2018,Douglas,Duren} for an introduction to frame theory,
dynamical sampling, and Hardy spaces.

In this paper, we delve deeper into this question. Specifically,
we show that if such a system forms a frame, then the symbol
$\phi$ must necessarly be an inner function, highlighting the
crucial role of inner functions in frame generation (Proposition
\ref{Prop.1}). However, this condition alone is not sufficient, as
the converse does not generally hold (Proposition \ref{Prop.2}).
Additional restrictions on the nature of $\phi$ and the vector $f$
reveal further structural constraints: constant inner functions
fail to generate frames through their powers, and the function $f$
must be cyclic within the disk to posses frame properties. For $\{
\phi^n f \}_{n=0}^\infty$ to constitute a frame for
$H^2(\mathbb{D})$, a geometric constraint must be satisfied by $f$
and $\phi$ (Proposition \ref{Prop.4}). Finally, a characterization
is provided that links the frame property of the system $\{
T_\phi^n f \}_{n=0}^\infty$ to the cyclicity of $f$ with respect
to $T_\phi$, offering a complete criterion for when such systems
form frames (Proposition \ref{Prop.6}). Collectively, these
results elucidate the subtle relationship between the spectral and
geometric aspects of multiplication operators and the analytic
features of their generating functions, thereby enhancing our
understanding of frame theory in Hardy spaces.

Throughout the paper, the results will be connected with the
operator theory, particularly as it relates to Hardy spaces and
multiplication operators appearing in applied harmonic analysis,
as well as with frame theory.\\

We need the following result in the sequel.
\begin{thm} (Beurling's theorem) \cite[Theorem 4.3, page
89]{Duren}\label{thm.1} Every nontrivial closed subspace
$M\subseteq H^2(\mathbb{D})$ that is invariant under
multiplication by $z$ has the form $M=\phi H^2(\mathbb{D})=\{\phi
f: f \in H^2(\mathbb{D})\}$, where $\phi$ is an inner function.
\end{thm}

\section{\textbf{Frame Properties of $\{T_{\phi}^nf\}_{n=0}^\infty=\{\phi^n f\}_{n=0}^\infty$ }}
In this section, we explore the relationship between
multiplication operators on the Hardy space $H^2(\mathbb{D})$ and
the properties of inner functions. Our primary focus lies in
understanding when the system $\{T^n f\}_{n=0}^\infty = \{\phi^n
f\}_{n=0}^\infty$ forms a frame for $H^2(\mathbb{D})$. We begin by
establishing that if the sequence $\{\phi^n f\}_{n=0}^\infty$ is a
frame for $H^2(\mathbb{D})$, then $\phi$ must be an inner
function. However, the converse is not necessarily true: not every
inner function $\phi$ yields such a frame.

\begin{prop} \label{Prop.1}
For any $\phi \in H^{\infty}(\mathbb{D})$, let $T_{\phi}:
H^2(\mathbb{D})\rightarrow H^2(\mathbb{D})$ be the associated
multiplication operator. If the system
$\{T_{\phi}^nf\}_{n=0}^{\infty}=\{ \phi^n f \}_{n=0}^\infty $ is a
frame for $H^2(\mathbb{D})$, then $\phi$ is an inner function.
\end{prop}
\begin{proof}
Suppose $\phi \in H^\infty(\mathbb{D})$ is not inner, meaning
$|\phi(e^{i\theta})|\neq 1$ on a set of positive measure.
Typically, for non-inner functions, $|\phi(e^{i\theta})|<1$ (since
$\|\phi\|_{H^{\infty}}\leq 1$ often in such contexts , or we
normalized), so we have $|\phi(e^{i\theta})|<1 \,\ \text{on a set
} E
\subset \mathbb{T} \text{ of positive measure}$.\\
Claim: $\|\phi ^n f\| \rightarrow 0$, as $n \rightarrow \infty$
for all $f \in H^2(\mathbb{D})$.

For the proof of claim, we not that the set $E$ satisfies $|E| >
0$ (i.e., $\phi$ is not unimodular a.e. on $\mathbb{T}$). Then for
all $n \geq 1$, $|\phi(e^{i\theta})|^n \to 0 \quad \text{uniformly
on } E$. Now let $f \in H^2(\mathbb{D})$, $f \neq 0$. Then the
boundary function $f(e^{i\theta})$ belongs to $L^2(\mathbb{T})$.
So
\begin{eqnarray*}
\|\phi^n f\|_{H^2}^2 = \int_0^{2\pi} |\phi(e^{i\theta})|^{2n}
|f(e^{i\theta})|^2 \, \frac{d\theta}{2\pi}.
\end{eqnarray*}
Now we split the integral over two parts:
\begin{itemize}
\item[(i)] On $E$, $|\phi(e^{i\theta})| < 1$, and the integrand
decays exponentially in $n$. \item[(ii)] On $\mathbb{T} \setminus
E$, $|\phi(e^{i\theta})| \leq 1$, so the integrand is bounded.
\end{itemize}
  Hence,
\begin{eqnarray*}
\|\phi^n f\|_{H^2}^2 = \int_E |\phi(e^{i\theta})|^{2n}
|f(e^{i\theta})|^2 \, \frac{d\theta}{2\pi} + \int_{\mathbb{T}
\setminus E} |\phi(e^{i\theta})|^{2n} |f(e^{i\theta})|^2 \,
\frac{d\theta}{2\pi}
\end{eqnarray*}
Let's denote:
\begin{eqnarray*}
I_n^{(1)} := \int_E |\phi(e^{i\theta})|^{2n} |f(e^{i\theta})|^2 \,
\frac{d\theta}{2\pi}, \quad I_n^{(2)} := \int_{\mathbb{T}
\setminus E} |\phi(e^{i\theta})|^{2n} |f(e^{i\theta})|^2 \,
\frac{d\theta}{2\pi},
\end{eqnarray*}
Then Since $|\phi(e^{i\theta})| < 1$ on $E$, there exists $r < 1$
such that $|\phi(e^{i\theta})| \leq r$ for a.e. $\theta \in E$.
Thus,
\begin{eqnarray*}
I_n^{(1)} \leq r^{2n} \|f\|_{L^2(E)}^2 \to 0 \quad \text{as } n
\to \infty.
\end{eqnarray*}
On $\mathbb{T} \setminus E$, $|\phi(e^{i\theta})| \leq 1$, so
$I_n^{(2)} \leq \|f\|_{L^2(\mathbb{T})}^2$, but can also go to
zero if $|\phi(e^{i\theta})| < 1$ on a set of positive measure.
But in any case,
\begin{eqnarray*}
\liminf_{n \to \infty} \|\phi^n f\|_{H^2}^2 \leq \lim_{n \to
\infty} I_n^{(1)} = 0.
\end{eqnarray*}
Therefore, $\|\phi^n f\|_{H^2} \to 0 \quad \text{as } n \to
\infty$. Of course, by dominated convergence theorem, also we
conclude
\begin{eqnarray*}
\lim_{n\rightarrow\infty}\|\phi^n f\|_{H^2}^2 = \int_0^{2\pi}
\lim_{n\rightarrow\infty}|\phi(e^{i\theta})|^{2n}
|f(e^{i\theta})|^2 \, \frac{d\theta}{2\pi}=0.
\end{eqnarray*}
Therefore, the vectors $\phi^n f$ are becoming arbitrarily small
in norm, and their span cannot remain "rich" enough to approximate
all of $H^2$, and hence cannot form a frame. For example, take $g
= f$. Then
\begin{eqnarray*}
\sum_{n=0}^\infty |\langle f, \phi^n f \rangle|^2 \leq \sum \|f\|
\|\phi^n f\|^2 = \|f\|^2 \sum \|\phi^n f\|^2 \to 0,
\end{eqnarray*}
since $\|\phi^n f\| \to 0$. This contradicts the existence of a
positive lower bound $A \|f\|^2$. More precisely, suppose
$\{\phi^n f\}_{n=0}^\infty$ is a frame for $H^2(\mathbb{D})$.
Then, in particular, there exists a lower frame bound $A>0$ such
that for all $g \in H^2(\mathbb{D})$,
\begin{eqnarray*}
\sum_{n=0}^\infty |\langle g, \phi^n f \rangle|^2 \geq A \|g\|^2.
\end{eqnarray*}
Let's test this inequality using some carefully chosen $g$. Let $g
= \phi^N f$ for some large $N$. Then
\begin{eqnarray*} \langle g, \phi^n f \rangle = \langle \phi^N f,
\phi^n f \rangle = \langle f, \phi^{n - N} f \rangle.
\end{eqnarray*}
 More concretely:
\begin{eqnarray*}
\sum_{n=0}^\infty |\langle \phi^N f, \phi^n f \rangle|^2 =
\sum_{n=0}^\infty |\langle f, \phi^{n - N} f \rangle|^2.
\end{eqnarray*}
This sum is dominated by $\|f\|^4$ and decays as $N \to \infty$,
because $f$ becomes orthogonal to $\phi^n f$ for large $n$, due to
decay of $\phi^n f$. Also, since $\|\phi^N f\| \to 0$, the
left-hand side becomes arbitrarily small, while the right-hand
side requires $A \|\phi^N f\|^2 \not\to 0$. Contradiction and
hence, the frame inequality fails if $\phi$ is not inner.
\end{proof}
By the concrete example, we show that if $\phi$ is not inner, then
$\{\phi^n f\}_{n=0}^\infty $ is not a frame.
\begin{exam}
Consider $\phi(z) = c$, where $0 < |c| < 1$, which is a constant
non-inner function. Then $\phi^n f = c^n f$. The frame condition
is
\begin{eqnarray*}
\sum_{n=0}^\infty |\langle g, c^n f \rangle|^2 = \sum_{n=0}^\infty
|c|^{2n} |\langle g, f \rangle|^2 = |\langle g, f \rangle|^2
\sum_{n=0}^\infty |c|^{2n}= \frac{|\langle g, f \rangle|^2}{1 -
|c|^2}.
\end{eqnarray*}
So frame inequality becomes
\begin{eqnarray*}
A \|g\|_{H^2}^2 \leq \frac{|\langle h, f \rangle|^2}{1 - |c|^2}
\leq B \|g\|_{H^2}^2,
\end{eqnarray*}
or
\begin{eqnarray*}
A (1 - |c|^2) \|g\|^2 \leq |\langle g, f \rangle|^2 \leq B (1 -
|c|^2) \|g\|^2.
\end{eqnarray*}
The inequality
\begin{eqnarray*}
A' \|g\|^2 \leq |\langle g, f \rangle|^2 \leq B' \|g\|^2
\end{eqnarray*}
(with $A' = A(1-|c|^2) > 0$ and $B' = B(1-|c|^2)$) means that the
linear functional
\begin{eqnarray*}
\Lambda_f: H^2(\mathbb{D}) \to \mathbb{C}, \quad \Lambda_f(g) =
\langle g, f \rangle
\end{eqnarray*}
is bounded below and above by multiples of $\|g\|$. The inequality
\begin{eqnarray*}\label{ineq.3}
|\langle g, f \rangle| \geq \sqrt{A'} \|g\|, \,\ \forall g \in
H^2(\mathbb{D}),
\end{eqnarray*}
for some $A'>0$, implies that the bounded linear functional
$\Lambda_f$ is injective and norm-equivalent to $\|g\|$. This can
only happen if the map $\Lambda_f:g \mapsto \langle g, f \rangle$
is an isomorphism between the space $H^2(\mathbb{D})$ and
$\mathbb{C}$. But this is impossible, because, the functional
$\Lambda_f:g \mapsto \langle g, f \rangle$ is rank 1, it maps all
of $H^2(\mathbb{D})$ into a 1-dimensional subspace of
$\mathbb{C}$, determined by the scalar values $\langle g, f
\rangle$. That is, the image of $\Lambda_f$ is just a line in
$\mathbb{C}$, and its kernel is
\begin{eqnarray*}
ker \Lambda_f = \{ g \in H^2(\mathbb{D}) : \langle g, f \rangle =
0 \} = f^\perp,
\end{eqnarray*}
the orthogonal complement of $f$. Since $H^2(\mathbb{D})$ is
infinite-dimensional, and $f^\perp$ is a closed hyperplane
(codimension one), there are infinitely many $g \neq 0$ such that
$\langle g, f \rangle = 0$, i.e., $\Lambda_f(g) = 0$. That is, for
any $g \in f^\perp$, we have $\langle g, f \rangle = 0 \Rightarrow
|\langle g, f \rangle|^2 = 0$, yet $\|g\|^2 > 0$. This makes $0
\geq A' \|g\|^2 > 0$, which is a contradiction. Hence, the lower
frame bound fails, so $\{\phi^n f\}_{n=0}^\infty$ cannot be a
frame.

Let's try a non-constant, non-inner function, such as $\phi(z) =
\frac{1}{2}z$. Then $\phi^n f = 2^{-n} z^n f$. The frame sum
becomes
\begin{eqnarray*}
\sum_{n=0}^\infty |\langle g, 2^{-n} z^n f \rangle|^2 =
\sum_{n=0}^\infty 4^{-n} |\langle g, z^n f \rangle|^2,
\end{eqnarray*}
the coefficients decay exponentially, and we need to check if this
spans $H^2(D)$. If $f(z) = 1$, then $\phi^n f = 2^{-n} z^n$, and
$\langle g, 2^{-n} z^n \rangle = 2^{-n}g_n$, where, $g_n:=\langle
g, z^n \rangle$, so
\begin{eqnarray*}
\sum_{n=0}^\infty |\langle g, 2^{-n} z^n \rangle|^2 =
\sum_{n=0}^\infty 4^{-n} |g_n|^2 .
\end{eqnarray*}
Compare to $\|g\|_{H^2}^2 = \sum_{n=0}^\infty |g_n|^2$. Since
${4^{-n}} \leq 1$ and decays rapidly, the sum $\sum_{n=0}^{\infty}
\frac{|g_n|^2}{4^n}$ may be smaller than $2\|g\|_{H^2}^2$. For
example, take $g(z) = z^k$, then $g_n=\delta_{nk}$, and
$\|g\|_{H^2}^2=1$. therefore we have
\begin{eqnarray*}
\sum_{n=0}^\infty |\langle g, 2^{-n} z^n \rangle|^2 =
\sum_{n=0}^\infty 4^{-n} |g_n|^2=\sum_{n=0}^\infty 4^{-n}|g_n|^2=
4^{-k},
\end{eqnarray*}
which tends to $0$ as $k\rightarrow \infty$, while $\|g\|_{H^2}=1$
remains constant. This violates the lower frame bound, i.e.,
\begin{eqnarray*}
A \|g\|_{H^2}^2 \leq \sum_|\langle g, \phi^n f\rangle|^2
\Rightarrow A\leq 4^{-k}\rightarrow 0 \,\ \mbox{as} \,\
k\rightarrow \infty.
\end{eqnarray*}
Hence the system is not a frame.
\end{exam}

The following result demonstrates that being inner does not
guarantee the orbit $\{\phi^n f \}_{n=0}^\infty $ forms a frame
for $H^2(\mathbb{D})$. Recall that, $f \in H^2(\mathbb{D})$ is
called a cyclic vector for $T_\phi$ if the span of $\{T_\phi^n
f\}_{n=0}^\infty = \{\phi^n f\}$ is dense in $H^2(\mathbb{D})$.
\begin{prop} \label{Prop.2}
The converse of the Prop. \ref{Prop.1} is not true in general; in
particular, for the classical inner function $\phi(z)=z^m, m\geq2$
there is no $f\in H^2(\mathbb{D})$ such that $\{\phi^n f
\}_{n=0}^\infty $ is a frame for $H^2(\mathbb{D})$.
\end{prop}
\begin{proof}
We show that, for any $f\in H^2(\mathbb{D})$, the orbit
$$\mathcal{F}:=\{T_{z^m}^n f \}_{n=0}^\infty=\{z^{mn} f
\}_{n=0}^\infty $$ can not be frame for $H^2(\mathbb{D})$. The
operator $T_\phi : H^2 \to H^2$, multiplication by $\phi = z^m$,
acts as $T_{z^m}(f)(z) = z^m f(z)$. Since $m \geq 2,\,\ T_\phi$ is
the $m-th$ power of the unilateral shift $S$ defined by $S(f)(z) =
z f(z)$. The unilateral shift $S$ on $H^2$ has invariant subspaces
that are described by Beurling's theorem (Theorem \ref{thm.1}),
they are exactly the subspaces $\theta H^2$ where $\theta$ is
inner. The operator $T_{z^m} = S^m$ is the $m-th$ power of the
unilateral shift. Consider the standard orthonormal basis
$\{z^n\}_{n=0}^\infty$ of $H^2(\mathbb{D})$. For $r = 0, 1, ...,
m-1$, define the closed subspaces
\begin{eqnarray*}
\mathcal{H}_r := z^r H^2(z^m)=\overline{\text{span}}\{ z^{r + km}
: k \ge 0 \},
\end{eqnarray*}
Where $H^2(z^m) = \{ g(z) = \sum_{k=0}^\infty c_k z^{km} :
\|g\|_{H^2} < \infty \}$  is the Hardy space over the variable
$z^m$. These subspaces are mutually orthogonal, and the space
$H^2(\mathbb{D})$ decomposes into
\begin{eqnarray*}
H^2 = \bigoplus_{j=0}^{m-1} \mathcal{H}_r,
\end{eqnarray*}
This is an orthogonal decomposition into reducing subspaces for
the operator $T_{z^m}$. In particular, $T_{z^m}$ leaves each
$\mathcal{H}_r$ invariant, i.e., because multiplying $z^{r+km}$ by
$z^m$ gives $z^m.z^{r+km}=z^{r+(k+1)m} \in H_r$. Thus,
$T_{z^m}H_r\subseteq H_r$, meaning $H_r$ are reducing subspaces
for $T_{z^m}$. The orbit $\mathcal{F}$ lies in a single subspace.
Let f $\in H^2(\mathbb{D})$ be arbitrary. Decompose it
\begin{eqnarray*}
f = \sum_{r=0}^{m-1} f_r, \quad f_r \in \mathcal{H}_r.
\end{eqnarray*}
Explicitly,
\begin{eqnarray*}
 f_r(z) = z^r f_r^{(0)}(z^m), \quad f_r^{(0)} \in
H^2(z^m).
\end{eqnarray*}
Then, for each $n$,
\begin{eqnarray*}
T_{z^m}^nf=z^{mn} f = \sum_{r=0}^{m-1} z^{mn} f_r=\sum_{j=0}^{m-1}
z^r \left( z^{m n} f_r^{(0)}(z^m) \right).
\end{eqnarray*}
Note that $z^{m n} f_r^{(0)}(z^m) \in H^2(z^m)$. Define the
subspace
\begin{eqnarray*}
\mathcal{M}_r := \overline{\mathrm{span}} \{ z^{m n}
f_r^{(0)}(z^m) : n \ge 0 \} \subseteq H^2(z^m).
\end{eqnarray*}
Then
\begin{eqnarray*}
\overline{\mathrm{span}} \mathcal{F} = \bigoplus_{r=0}^{m-1} z^r
\mathcal{M}_r.
\end{eqnarray*}
The space $H^2(z^m)$ is isometrically isomorphic to the classical
Hardy space $H^2(\mathbb{D})$, just with variable change $w =
z^m$. Also, $f_r^{(0)}$ is cyclic for multiplication by $z^m$ in
$H^2(z^m)$ if and only if the cyclic subspace $\mathcal{M}_r =
H^2(z^m)$. If $f_r^{(0)}$ is not cyclic,
then $\mathcal{M}_r$ is a proper subspace of $H^2(z^m)$.\\
For $\overline{\mathrm{span}} \mathcal{F} = H^2$, we must have
\begin{eqnarray*}
\bigoplus_{r=0}^{m-1} z^r \mathcal{M}_r =\bigoplus_{j=0}^{m-1}
\mathcal{H}_r= \bigoplus_{r=0}^{m-1} z^r H^2(z^m),
\end{eqnarray*}
so $\mathcal{M}_r = H^2(z^m)$ for all $r$. Thus, each $f_r^{(0)}$
must be cyclic in $H^2(z^m)$. But $\{\phi^n f\}$ is a single orbit
under multiplication by $z^m$, in fact, the family $\mathcal{F}$
is generated by powers of the single operator $T_\phi$ applied to
$f$. This family cannot separate the components in $\mathcal{H}_r$
because multiplication by $z^m$ leaves each $\mathcal{H}_r$
invariant. So the orbit $\{\phi^n f\}$ only generates cyclic
subspaces inside each $\mathcal{H}_r$. Therefore, the closed span
of $\mathcal{F}$ is a proper closed subspace of $H^2$, so
$\mathcal{F}$ cannot be a frame for $H^2$.
\end{proof}

The following result shows that when $\phi$ is a constant inner
function, the system $\{\phi^n f\}_{n=0}^\infty$ fails to be a
frame for $H^2(\mathbb{D})$.

\begin{prop} \label{Prop.3}
If $\phi$ is a constant inner function, then $\{\phi^n f
\}_{n=0}^\infty $ cannot be a frame for $H^2(D)$.
\end{prop}
\begin{proof}
Suppose $\phi(z) = c$ with $|c| = 1$, then $\phi^n f = c^n f$, and
\begin{eqnarray*}
\langle g, \phi^n f \rangle = \langle g, c^n f \rangle = \bar{c}^n
\langle g, f \rangle.
\end{eqnarray*}
The frame condition becomes
\begin{eqnarray*}
\sum_{n=0}^\infty |\langle g, \phi^n f \rangle|^2 =
\sum_{n=0}^\infty |\bar{c}^n \langle g, f \rangle|^2 = |\langle g,
f \rangle|^2 \sum_{n=0}^\infty 1.
\end{eqnarray*}
So the upper bound $B\|g\|_2^2$ cannot hold unless $\langle h, f
\rangle = 0$ for all $h\in H^2(\mathbb{D})$, implying $f=0$, which
again implies the system does not span $H^2(\mathbb{D})$. Thus, a
constant $\phi$ does not yield a frame unless $f$ is trivial,
which contradicts the frame spanning $H^2(\mathbb{D})$.
\end{proof}
\begin{cor}
If $|\phi(z)| = 1$ for all $z \in \mathbb{D}$, then $\{\phi^n f
\}_{n=0}^\infty $ cannot be a frame for $H^2(\mathbb{D})$.
\end{cor}
\begin{proof}
By the maximum modulus principle, $\phi$ must be constant, say
$\phi(z) = c$ with $|c| = 1$. By Prop. \ref{Prop.3}, $\{\phi^n f
\}_{n=0}^\infty$ cannot be a frame for $H^2(\mathbb{D})$.
\end{proof}
We have a geometric constraint on $\phi$: the image
$\phi(\mathbb{D})$ must intersect the unit circle $\mathbb{T}$.
This connection between the frame property and the boundary
behavior of $\phi$ reveals a deeper structure governing
multiplication operators on Hardy spaces.
\begin{prop} \label{Prop.4}
If $\{\phi^n f\}_{n=0}^\infty$ is a frame for $H^2(\mathbb{D})$,
then
\begin{itemize}
\item[(i)] $\phi(\mathbb{D})\cap \mathbb{T}\neq \emptyset$.
\item[(ii)] The function $f$ has no zeros inside $\mathbb{D}$.
\end{itemize}
\end{prop}
\begin{proof}
(i) Assume, for contradiction, that $\phi(\mathbb{D})\cap
\mathbb{T} = \varnothing$. Since $\phi$ is analytic and bounded on
$\mathbb{D}$, and since the unit circle $\mathbb{T}$ is closed and
compact, the image $\phi(D)$ is a connected open subset of the
complex plane (by the open mapping theorem), and it does not
intersect $\mathbb{T}$. This implies either
\begin{itemize}
\item[(a)] $\sup_{z \in D} |\phi(z)| = \|\phi\|_\infty = r < 1$.
\item[(b)] $\inf_{z \in D} |\phi(z)| = s > 1$.
\end{itemize}
Because the image is connected and does not intersect the unit
circle, the entire image lies strictly inside the open disk of
radius $1$ or strictly outside the closed disk of radius $1$.

Let $\|\phi\|_\infty = r < 1$. Since $\phi$ is bounded by $r < 1$,
the multiplication operator $T_\phi$ is a strict contraction on
$H^2(D)$. Specifically, $\|T_\phi\| = \|\phi\|_\infty = r < 1$.
Then,
\begin{eqnarray*}
\|T_\phi^n f\| = \|\phi^n f\| \leq \|\phi\|_\infty^n \|f\| = r^n
\|f\| \to 0 \quad \text{as } n \to \infty.
\end{eqnarray*}

The sequence $\{\phi^n f\}$ decays exponentially to zero in norm.
Now consider the frame inequality
\begin{eqnarray*}
A \|g\|^2 \leq \sum_{n=0}^\infty |\langle g, \phi^n f \rangle|^2
\leq B \|g\|^2, \,\ \forall g \in H^2(\mathbb{D}),
\end{eqnarray*}
for some $A, B > 0$. Because $\phi^n f$ tend to zero in norm, its
contradicting the frame lower bound $A > 0$.  Thus $\{\phi^n
f\}_{n=0}^\infty$ cannot be a frame.\\ Now let $\inf_{z \in D}
|\phi(z)| = s > 1$, i.e., the image of $\phi$ is strictly outside
the unit circle, then $|\phi(z)| \geq s > 1$ for all $z \in D$.
Then,
\begin{eqnarray*}
\|T_\phi^n f\| = \|\phi^n f\| \geq s^n \|f\|,
\end{eqnarray*}
which grows exponentially in $n$. Since $T_\phi^n f \in H^2(D)$,
but the norms grow unbounded, this contradicts the boundedness of
the powers on $H^2$. Furthermore, a frame must satisfy an upper
frame bound $B$, so $\sum_n \|\phi^n f\|^2 < \infty$ cannot hold
if $\|\phi^n f\|$ grows exponentially. Contradiction. The only
remaining possibility is that $\phi(D)$ must intersect the unit
circle $\mathbb{T}$, i.e., $\phi(D) \cap \mathbb{T} \neq
\varnothing$, as desired.

For (ii), Assume $f(z_0) = 0$ for some $z_0 \in \mathbb{D}$. Then
for all $n, \,\ (\phi^n f)(z_0) = \phi(z_0)^n f(z_0) = 0$. So
every vector in  $\{\phi^n f\}_{n=0}^\infty$ vanishes at $z_0$.
Let $K_{z_0} \in H^2(\mathbb{D})$ be the reproducing kernel at
$z_0$, defined by
\begin{eqnarray*}
K_{z_0}(z) = \frac{1}{1 - \overline{z_0} z}.
\end{eqnarray*}
The evaluation functional
\begin{eqnarray*}
g \mapsto g(z_0) = \langle g, K_{z_0} \rangle_{H^2},
\end{eqnarray*}
is bounded and nonzero. Since all $\phi^n f$ vanish at $z_0$,
$$\langle \phi^n f, K_{z_0} \rangle = (\phi^n f)(z_0) = 0.$$ This
implies $K_{z_0}$ is orthogonal to the closed linear span of
$\{\phi^n f\}_{n=0}^\infty$. Hence
$\overline{span}\{\phi^nf\}_{n=0}^\infty$ is a proper subspace of
$H^2$, and the system $\{\phi^n f\}_{n=0}^\infty$ does not span
all of $H^2(\mathbb{D})$. Therefore, it cannot be a frame for
$H^2(\mathbb{D})$.
\end{proof}
The following section shows that the frame property of
$\{T_{\phi}^nf\}_{ n=0 }^{\infty}=\{\phi^n f\}_{n=0}^\infty$
depends heavily on $f$.
\section{\textbf{Non-constant Inner Function}}

We ask under what condition does $\{\phi^n f \}_{n=0}^\infty$ form
a frame for $H^2(\mathbb{D})$, and what does this imply about
$\phi$ and $f$?

The function $f \in H^2(\mathbb{D})$ is critical. If $f$ is cyclic
for $\phi$ and if the frame bounds hold, then the sequence could
be a frame.
\begin{exam}\label{example.3-1}
Consider a simple inner function $\phi(z)=z$. This is
non-constant, analytic in $\mathbb{D}$, and satisfies
$|\phi(z)|=|z|<1$ in $\mathbb{D}$, but $|\phi(e^{i\theta})|= 1$
a.e. on $\mathbb{T}$, consistent with being an inner function. The
system becomes $T_\phi^n f = z^n f$. Suppose $f(z)=1$ (the
constant function, which is in $H^2(\mathbb{D}))$. Then, $\phi^n
f=z^n$. The system $\{z^n\}_{n=0}^\infty$ is the standard
orthonormal basis for $H^2(\mathbb{D})$. For any $h \in
H^2(\mathbb{D})$, with $h(z) = \sum_{k=0}^\infty a_k z^k$,
\begin{eqnarray*}
\langle h, z^n \rangle_{H^2} = \int_0^{2\pi} \left(
\sum_{k=0}^\infty a_k e^{i k \theta} \right) e^{-i n \theta}
\frac{d\theta}{2\pi} = a_n,
\end{eqnarray*}
 since
$\int_0^{2\pi} e^{i (k-n) \theta} \frac{d\theta}{2\pi} =
\delta_{k,n}$. Thus,
\begin{eqnarray*}
\sum_{n=0}^\infty |\langle h, z^n \rangle|^2 = \sum_{n=0}^\infty
|a_n|^2 = \|h\|_{H^2}^2.
\end{eqnarray*}
This satisfies the frame condition with $A = B = 1$, so
$\{z^n\}_{n=0}^\infty$ is a tight frame (in fact, an orthonormal
basis). This example suggests that for some $f$ (e.g., $f = 1$)
and some non-constant $\phi$ with $|\phi| = 1$ on the boundary,
the system can be a frame.

Now, consider a general $f \in H^2(\mathbb{D})$. The system $\{z^n
f\}$ may not be a frame for all $f$. For example, let's take $f(z)
= 1-z$. Then, $z^n f(z) = z^n (1 - z) = z^n - z^{n+1}$. We need to
check if $\{z^n - z^{n+1}\}_{n=0}^\infty$ is a frame. Because
$\{z^n\}_{n=0}^\infty$ is an orthonormal basis for
$H^2(\mathbb{D})$, inner product satisfy $\langle h, z^n
\rangle=a_n$, and so we have
\begin{eqnarray*}
\langle h, z^n - z^{n+1} \rangle = \langle h, z^n \rangle -
\langle h, z^{n+1} \rangle = a_n - a_{n+1},
\end{eqnarray*}
where $h(z) = \sum_{k=0}^\infty a_k z^k$. The frame sum is

\begin{eqnarray*}
\sum_{n=0}^\infty |\langle h, z^n - z^{n+1} \rangle|^2 =
\sum_{n=0}^\infty |a_n - a_{n+1}|^2.
\end{eqnarray*}
To determine if this satisfies the frame condition, consider the
lower bound. Suppose $h$ is such that $a_n = a_{n+1}$ for all $n$,
e.g., $h(z) = \frac{1}{1-z} = \sum_{n=0}^\infty z^n$ (which is in
$H^2(\mathbb{D})$ in a weak sense, but let's consider a smoothed
version if needed). Then, $a_n = 1$, so $a_n - a_{n+1} = 1 - 1 =
0$, and
\begin{eqnarray*}
\sum_{n=0}^\infty |a_n - a_{n+1}|^2 = 0,
\end{eqnarray*}
while $\|h\|_{H^2}^2 = \sum_{n=0}^\infty |a_n|^2 = \infty$, which
would contradiction the lower frame bound.
\end{exam}


we examine when the family $\{T^n f\}_{n=0}^\infty$ forms a frame
for $H^2(D)$. It turns out that this occurs if and only if $f$ is
a cyclic vector for $T$, meaning the closed linear span of $\{T^n
f : n \geq 0\}$ is the entire space $H^2(D)$. Consequently, if $f$
is non-cyclic, the sequence $\{T^n f\}_{n=0}^\infty$ fails to
constitute a frame, highlighting the essential role of cyclicity
in frame theory within Hardy spaces.
\begin{prop} \label{Prop.6}
Let $\phi \in H^{\infty}(\mathbb{D})$ be an inner function, and
let $T_{\phi}: H^2(\mathbb{D})\rightarrow H^2(\mathbb{D})$ be the
associated multiplication operator. Then
$\{T^n_{\phi}f\}_{n=0}^{\infty}=\{\phi^nf\}_{n=0}^{\infty}$ is a
frame for $H^2(\mathbb{D})$ if and only if $f$ is a cyclic vector
for $T_{\phi}$. In particular, if $f$ is non-cyclic, then
$\{\phi^n f\}_{n=0}^{\infty}$ cannot be a frame for
$H^2(\mathbb{D})$.
\end{prop}
\begin{proof}
Suppose $\{\phi^nf\}_{n=0}^{\infty}$ is a frame for
$H^2(\mathbb{D})$. By definition, there exists constants $A,B > 0$
such that

\begin{eqnarray*}
A \|g\|_2^2 \leq \sum_{n=0}^\infty |\langle g, \phi^n f \rangle|^2
\leq B \|g\|_2^2, \,\ \forall g \in H^2(\mathbb{D})
\end{eqnarray*}
The lower bound implies that if $\langle g, \phi^n f\rangle=0$ for
all $n$, then $g=0$, meaning that the sequence is complete. i.e.,
the closed linear span of $\{\phi^nf\}_{n=0}^{\infty}$ must be the
whole space $H^2(\mathbb{D})$,
$$\overline{span}\{\phi^nf\}_{n=0}^{\infty}=H^2(\mathbb{D}),$$
which means that $f$ is cyclic for $T_{\phi}$. Now suppose $f$ is
cyclic for $T_{\phi}$, means the synthesis operator
\begin{eqnarray*}
X: \ell^2(\mathbb{N}_0) \to H^2(\mathbb{D}), \quad X(c) =
\sum_{n=0}^\infty c_n \varphi^n f,
\end{eqnarray*}
for any $c=(c_n)\in \ell^2(\mathbb{N}_0)$, is surjective. Because
\begin{eqnarray*}
\| X(c_n) \| = \left\| \sum_{n=0}^\infty c_n \varphi^n f \right\|
\leq \|f\| \cdot \left\| \sum_{n=0}^\infty c_n \varphi^n
\right\|_{2},
\end{eqnarray*}
and since $\phi$ is bounded, multiplication by $\phi$ is bounded
on $H^2(\mathbb{D})$, so the operator $X$ is bounded. The analysis
operator (the adjoint of $X$) is
\begin{eqnarray*}
X^*: H^2(\mathbb{D}) \to \ell^2(\mathbb{N}_0), \quad X^*(g) =
(\langle g, \varphi^n f \rangle)_{n=0}^\infty.
\end{eqnarray*}
The frame condition for $\{\phi^n f\}_{n=0}^{\infty}$ hold if and
only if there exist constants $A,B>0$ such that for all $g \in
H^2(\mathbb{D})$,
\begin{eqnarray*}
A \|g\|^2 \leq \sum_{n=0}^\infty |\langle g, \phi^n f \rangle|^2 =
\|X^* g\|_{\ell^2}^2\leq B\|g\|^2.
\end{eqnarray*}
This is equivalent to $X$ being a bounded and bounded below
operator, or equivalently, $X^*$ being bounded and having bounded
inverse on its range. Since $X$ is a bounded surjective operator
between Hilbert spaces, by the Open Mapping Theorem, $X$ is an
open map. Hence, there exists $A>0$ such that for every $g \in
H^2(\mathbb{D})$, there exists $c = (c_n) \in \ell^2$ with
\begin{eqnarray*}
X(c) = g, \quad \text{and} \quad \|X(c)\|_{2} \geq A
\|c\|_{\ell^2}.
\end{eqnarray*}
i.e., $X$ is bounded below. Since $X$ is onto and bounded below,
$X^*$ is bounded below on $H^2(\mathbb{D}), i.e., \|X^*
g\|_{\ell^2} \geq A \|g\|$. The upper bound $\|X^* g\|_{\ell^2}
\leq \|X\| \|g\|$ is automatic because $X^*$ is bounded. This
gives the frame bounds:
\begin{eqnarray*}
A^2 \|g\|^2 \leq \sum_{n=0}^\infty |\langle g, \varphi^n f
\rangle|^2 \leq \|X\|^2 \|g\|^2.
\end{eqnarray*}
Hence, $\{ \varphi^n f \}$ is a frame. If $f$ is not cyclic, then
\begin{eqnarray*}
M:=\overline{span}\{T^n_{\phi}f\}_{n=0}^{\infty}=\overline{span}\{\phi^nf\}_{n=0}^{\infty}\varsubsetneq
H^2(\mathbb{D})
\end{eqnarray*}
is a proper closed $T_{\phi}$-invariant subspace of
$H^2(\mathbb{D})$, and the system $\{ \varphi^n f
\}_{n=0}^{\infty}$ cannot cover the wholes space, so the frame
inequalities cannot hold on all of $H^2(\mathbb{D})$, because
vectors orthogonal to $M$ have zero coefficients in the frame
expansions, i.e.,
\begin{eqnarray*}
\sum_{n=0}^{\infty}|\langle g, T^n_{\phi}f \rangle|^2=0, \,\
\mbox{for}\,\ g\perp M,\,\ g\neq 0,
\end{eqnarray*}
violating the lower frame bound condition.
\end{proof}
\bigskip

\section{Acknowledgment}
The author thanks to artificial intelligence, for improving the
language and assisting in the proof of Example \ref{example.3-1}.
\noindent


\bibliographystyle{plain}
\end{document}